\documentclass[12pt]{amsart}
\usepackage{mathptmx,eucal,amsmath,amscd,amssymb,amsthm,xspace, diagrams}
\usepackage{hyperref,amssymb,color, url,fancyhdr}
\usepackage[cp850]{inputenc}
\usepackage{latexsym}

\usepackage{graphics}
\usepackage{array}

\definecolor{red}{rgb}{1,0,0}

\theoremstyle{plain}

\newtheorem{thm}{Theorem}

\newtheorem{conj}{Conjecture}

\newtheorem{obs}[thm]{Observation}
\newtheorem*{namedtheorem}{\theoremname}

\newcommand{\theoremname}{testing}

\theoremstyle{definition}

\newtheorem{rem}[thm]{Remark}

\def\N{\mathbb{N}}

\def\hat{\widehat}

\DeclareMathOperator{\Alt}{Alt}

\newcommand{\intfunc}{i}

\newcommand{\CG}{\mathcal G} %\newcommand{\CH}{\mathcal H}

\newcommand{\comment}[1]{}

\fancypagestyle{plain}

\begin{document}
\bibliographystyle{plain}

%-----------------------------------------------------------
%-----------------------------------------------------------

\title{\textbf{On groups with slow intersection growth}}
\author{Martin Kassabov \and Francesco Matucci}
\maketitle

%-----------------------------------------------------------
%-----------------------------------------------------------
%---------------------Abstract---------------------------

\begin{abstract}
Intersection growth concerns the asymptotic behavior of
the index of the intersection of all subgroups of a group that have index at most $n$.
In this note we show that the intersection growth of some groups may not be a nicely behaved function
by showing the following seemingly contradictory results:
(a) for any group $G$ the intersection growth function $i_G(n)$  is super linear infinitely often; and
(b) for any increasing function $f$ there exists a group $G$ such that
$i_G$ below $f$ infinitely often.
\end{abstract}

%\nid keywords: \emph{The people's keywords.}

%-----------------------------------------------------------
%-----------------------------------------------------------
\section{Introduction}

Let $\CG$ be a class of subgroups of a group $\Gamma$.
We define the $\CG$-\emph{intersection growth function} of $\Gamma $ by letting $\intfunc^\CG_\Gamma(n)$ be the index of the intersection of all $\CG$-subgroups of $\Gamma $ with index at most $n $.   In symbols,
$$
\intfunc^\CG_\Gamma(n) := [\Gamma : \Lambda^\CG_\Gamma(n)]
,\ \ \  \text { where } \ \Lambda^\CG_\Gamma(n) := \bigcap_{[\Gamma: \Delta] \leq n, \Delta \in \CG} \Delta.
$$
Here, $\CG $ will always be either the class of all subgroups, the class $\lhd$ of normal subgroups or the class $\max \lhd $ of maximal normal subgroups of $\Gamma $, i.e.\ those subgroups that are maximal among normal subgroups.  The corresponding intersection growth functions will then be written $i_{\Gamma}  (n)$, $i_{\Gamma} ^\lhd (n) $ and $i_{\Gamma} ^ {\max\lhd} (n) $.   %The former clearly grows at least as fast as the latter, but in general their asymptotics can be very different.

Intersection growth has been first defined by Biringer, Bou-Rabee and the authors
in~\cite{matucci12} where it has been studied for free groups and some arithmetic groups
and a connection has been drawn between intersection growth and the residual finiteness
growth, which will be mentioned at the end of this introduction.

The aim of this note is to build examples of groups whose intersection growth behaves
slowly at certain integers and much faster at others.
We will assume all groups in this note to be finitely generated. For a group $\Gamma$,
let $R(\Gamma)$ denote the intersection of all finite index subgroups of $\Gamma$. Obviously,
$i^\CG_\Gamma(n)=i^{\mathcal{G}}_{\Gamma/R(\Gamma)}(n)$, so we can assume that $R(\Gamma)=\{1\}$,
i.e.
\emph{we will assume $\Gamma$ to be finitely generated and residually finite}.
%The following is an easy observation and will be proved in
%Section~\ref{sec:proofs}.

\begin{obs}\label{thm:easy-observation}
For any infinite residually finite group $\Gamma$ 
%For any finitely generated residually finite group $\Gamma$ 
one has that $i_\Gamma(n) \geq n$ for infinitely many
positive integers $n$.
\end{obs}

The main result of this note states that, in some sense, the  opposite is also true.

\begin{thm}
\label{thm:main-theorem}
For any strictly increasing function $f: \N \to \N$ there exists a finitely generated
residually finite group $\Gamma$
such that $i_{\Gamma}(n) < f(n)$ for infinitely many positive integers $n$.
\end{thm}

\begin{rem}
We observe that Observation~\ref{thm:easy-observation} and Theorem~\ref{thm:main-theorem} are still true if we replace $i_{\Gamma}(n)$ with either
$i_{\Gamma}^{\lhd}(n)$ or $i_{\Gamma}^{\max\lhd}(n)$. The proofs are easy adaptations
of the ones we give below and we omit them to keep this note short.
\end{rem}

Observation~\ref{thm:easy-observation} and Theorem~\ref{thm:main-theorem} show
that intersection growth functions may behave in a bad way. More precisely,
these functions cannot be approximated via regular functions
(polynomial, exponential, etc.) as they have different behavior at different integers.

The main idea is to construct a group with a very limited number of finite quotients and this is achieved by
taking a suitable direct product of a family of simple groups.

Recall that for a residually finite finitely generated
group $\Gamma=\langle X \rangle$, one can define the \emph{residual finiteness growth function}
$F_{\Gamma,X}(n)$
as the minimal natural number $N$ such that any element of word length $\le n$
with respect to $X$ can be detected in a quotient $Q$ of cardinality $\le N$
(see Bou-Rabee~\cite{bourabee1} for an introduction).
All the constructions in this paper can immediately be translated to get similar information about
residual finiteness growth. See Bou-Rabee and Seward~\cite{bourabee-seward1}
for a similar result about building examples with
arbitrarily large residual finiteness growth functions.
We also mention the work of Kharlampovich, Myasnikov and Sapir~\cite{kharl-myasni-sapir-1}
who build finitely presented solvable examples with arbitrarily large
residual finiteness growth functions.

Finally, we recall that for a group
$\Gamma$, one defines $r_{\Gamma}(n)$ to be the number of
inequivalent finite irreducible $n$-dimensional representations of $G$
(assuming this number is finite). One defines $R_{\Gamma}(n) = \sum_{k=1}^n r_{\Gamma}(k)$
to be the \emph{representation growth function} of $\Gamma$.
Craven~\cite{craven1} showed analogues of
Observation~\ref{thm:easy-observation} and Theorem~\ref{thm:main-theorem}
for $R_{\Gamma}(n)$.
Unlike Observation~\ref{thm:easy-observation},
Craven's upper bound Theorem G requires a more careful and involved construction.

\section{Proofs of the results}
\label{sec:proofs}

\begin{proof}[Proof of Observation~\ref{thm:easy-observation}]
Since $\Gamma$ is infinite and residually finite for every $n$ there exists a finite index subgroup $H_N < \Gamma$ of index 
$[\Gamma:H_N] = M > n$. This clearly implies that $\Lambda_\Gamma(N) \leq H_N$, i.e., $i_{\Gamma}(N) \geq N$.
\iffalse 
We argue by contradiction. Without loss of generality,
assume that $\Gamma$ is such that $i_\Gamma(n) \le n$.
There is an infinite chain
of finite index normal subgroups $\Gamma :=H_0 > H_1 > H_2 > \ldots$.
Since the quotient
$\Gamma/R(\Gamma)=\bigcup_{n \in \mathbb{N}} \Gamma/H_i$
is finitely generated, then $\Gamma/R(\Gamma)=\Gamma/H_s$ for some $s$ and
it is finite, which is impossible since $\Gamma/R(\Gamma) \cong \Gamma$
is infinite.
\fi
\end{proof}

\begin{obs}
\label{thm:product-alternating-groups}
Let $\{n_i\}$ be a strictly increasing sequence of positive integers, such that $n_1 \geq 5$.
If $\mathfrak{G} := \prod_{i=1}^{\infty} \Alt(n_i)$, then
\[
i_\mathfrak{G}(n_k-1)= \prod_{i<k} |\Alt(n_k)|.
\]
\end{obs}

\begin{proof}
The group $\mathfrak{G}$ is profinite and endowed with the product topology.
Notice that for profinite groups is natural to define intersection growth with respect to
finite index closed subgroups. In this case by
Nikolov and Segal~\cite{nikolov-segal-1}
every subgroup of finite index is closed, but we do need to use this
fact. We observe the following:
\begin{enumerate}
\item if $H \le \mathfrak{G}$ has index $\le n_k-1$,
then $H \cap \Alt(n_k)$ has index $\le n_k-1$ inside $\Alt(n_k)$,
\item $\Alt(n_k)$  has no nontrivial subgroups
of index $\le n_k-1$, but it has subgroup of index $n_k$.
%If it did, it would have a nontrivial normal subgroup of index $\le (n_k-1)!$.
\end{enumerate}
Therefore, if $H \le \mathfrak{G}$ is a closed subgroup of  index $\le n_k-1$, then $H \cap \Alt(n_i)=\Alt(n_i)$,
for any $i \ge k$. This implies that
$$
H \ge \overline{\bigoplus_{i\ge k} \Alt(n_i)} = \prod_{i\ge k} \Alt(n_i)
$$
and therefore
$\Lambda_{\mathfrak{G}}(n_k-1) \ge \prod_{i\ge k} \Alt(n_i)$.
It is very easy to see the opposite inclusion
(since the group $\prod_{i\ge k} \Alt(n_i)$ can be realized as the intersection
of point stabilizers of the actions of $\mathfrak{G}$ via $\Alt(n_i)$ on $n_i$ points for $i \leq k$), 
which concludes the proof.
\end{proof}

This observation shows that if the sequence $\{n_k\}$ grows very quickly then, in some sense,
the function $i_\mathfrak{G}$ is very small at certain values.

\begin{rem}
The group $\mathfrak{G}$ constructed in Observation~\ref{thm:product-alternating-groups}
is not finitely generated, but it is finitely generated as a profinite group.
\end{rem}

The idea is to construct a finitely generated group whose profinite completion is the same as
the group $\mathfrak{G}$ constructed in Observation~\ref{thm:product-alternating-groups}.
Since $\mathfrak{G}$ is finitely generated as a profinite group it is very easy to find finitely generated dense subgroup $\Gamma \subseteq \mathfrak{G}$. Such an embedding will give a surjective map $\pi: \hat \Gamma \to  \mathfrak{G}$, however this map is often not injective because $\Gamma$ might have finite index subgroups which which are not visible in $\mathfrak{G}$~\cite{pyber1,Neumann1} and can not be used to obtain information about $i_\Gamma$. If the map $\pi$ is an isomorphism it is very easy to see that
$i_\Gamma = i_\mathfrak{G}$ (see~\cite{matucci12}).
%\textcolor{red}{(CITE) (THE CORRECT REMARK)}
%\textcolor{red}{This is an obvious observation which has to appear soon in the paper
%\cite{matucci12} and relies on the fact that it's a standard fact that there is a bijection $\varphi$
%between
%finite index subgroups of $\Gamma$ and finite index closed subgroups of $\widehat{\Gamma}$.
%This bijection is nice enough to
%respect all indices of $H$ and $\varphi(H)$ within the respective ambient groups
%and also respects the indices of the intersection of such subgroups.}

For some time it was not known if it is possible to find $\Gamma$ such that the map $\pi$ is an isomorphism. This question was settled in 
Theorem~1.2 in~\cite{kas-niko-1},
which we restate here in the weaker form we need.

\begin{thm}[Kassabov-Nikolov, \cite{kas-niko-1}]
\label{thm:weak-kassabov-nikolov}
For any strictly increasing sequence $\{n_k\}$,
the Cartesian product
$\prod_{k=1}^{\infty} \Alt(n_k)$ is a profinite completion
of a finitely generated residually finite group.
\end{thm}

We are now ready to prove Theorem~\ref{thm:main-theorem}.

\begin{proof}[Proof of Theorem~\ref{thm:main-theorem}]
We choose $n_1=5$ and, inductively, we define $n_k$ in the following way:
given $n_k$, let $n_{k+1}$ be an integer large enough such that
$\prod_{i \le k} |\Alt(n_k)| < \min\{n_k+1, f(n_{k+1}-1)\}$.

The chosen sequence allows us to construct the group $\mathfrak{G}$ of
Observation~\ref{thm:product-alternating-groups}.
By Theorem~\ref{thm:weak-kassabov-nikolov} we have that $\mathfrak{G}=\widehat{\Gamma}$,
for a finitely generated residually finite group $\Gamma$.
As was observed above
$i_{\Gamma}=i_{\mathfrak{G}}$ and so
$i_{\Gamma}(n_k-1)= \prod_{i<k} |\Alt(n_k)|$.
By construction
$i_{\Gamma}(n_k-1) < f(n_k-1)$ for any $k>1$ and we are done.
\end{proof}

\section{Acknowledgements}
The first author was partially supported by NSF grants DMS-0900932 and DMS-1303117.
The second author gratefully acknowledges the Fondation Math\'ematique Jacques Hadamard (FMJH - ANR - Investissement d'Avenir) for the support received during the development of this work.
We are grateful to Khalid Bou-Rabee for helpful mathematical conversations.

\bibliography{go}

%\bigskip

\noindent
Martin Kassabov \\
%Department of Mathematics, University of Southampton \\
%University Road, Southampton S017 1BJ, UK \\
%and\\
Department of Mathematics, Cornell University\\
Malott Hall, Ithaca, NY 14850, USA\\
E-mail: kassabov@math.cornell.edu\\

\noindent
Francesco Matucci \\
D\'epartement de Math\'ematiques, Facult\'e des Sciences d'Orsay, \\
Universit\'e Paris-Sud 11, B\^atiment 425, Orsay, France \\
E-mail: francesco.matucci@math.u-psud.fr
%-----------------------------------------------------------
%-----------------------------------------------------------

\end{document}